\documentclass[twoside,english]{article}
\usepackage[T1]{fontenc}
\usepackage[latin9]{luainputenc}
\usepackage{geometry}
\usepackage{babel}
\usepackage{amsmath}
\usepackage{amsthm}
\usepackage{amssymb}
\usepackage[unicode=true,pdfusetitle,
 bookmarks=true,bookmarksnumbered=false,bookmarksopen=false,
 breaklinks=false,pdfborder={0 0 1},backref=false,colorlinks=false]
 {hyperref}

\makeatletter
\numberwithin{equation}{section}
\numberwithin{figure}{section}
\theoremstyle{plain}
\newtheorem{thm}{\protect\theoremname}[section]
\theoremstyle{remark}
\newtheorem{rem}[thm]{\protect\remarkname}
\theoremstyle{plain}
\newtheorem{prop}[thm]{\protect\propositionname}
\theoremstyle{plain}
\newtheorem{lem}[thm]{\protect\lemmaname}

\@ifundefined{date}{}{\date{}}
\usepackage[nottoc,notlot,notlof]{tocbibind}

\makeatother

\providecommand{\lemmaname}{Lemma}
\providecommand{\propositionname}{Proposition}
\providecommand{\remarkname}{Remark}
\providecommand{\theoremname}{Theorem}

\begin{document}
\title{\textbf{Classical solutions to local first-order extended mean field
games}}
\author{Sebastian Muñoz}
\maketitle
\begin{abstract}
We study the existence of classical solutions to a broad class of
local, first order, forward-backward extended mean field games systems,
that includes standard mean field games, mean field games with congestion,
and mean field type control problems. We work with a strictly monotone
cost that may be fully coupled with the Hamiltonian, which is assumed
to have superlinear growth. Following previous work on the standard
first order mean field games system, we prove the existence of smooth
solutions under a coercivity condition that ensures a positive density
of players, assuming a strict form of the uniqueness condition for
the system. Our work relies on transforming the problem into a partial
differential equation with oblique boundary conditions, which is elliptic
precisely under the uniqueness condition.
\end{abstract}
\textit{\small{}MSC: 35Q89, 35B65, 35J66, 35J70. }{\small\par}

\noindent \textit{\small{}Keywords: quasilinear elliptic equations;
oblique derivative problems; Bernstein method; non-linear method of
continuity; Hamilton-Jacobi equations.}{\small\par}

\noindent 
\global\long\def\xo{\bar{x}}%

\global\long\def\tr{\text{{Tr}}}%

\global\long\def\utilde{\tilde{u}}%

\global\long\def\diag{\text{{diag}}}%

\global\long\def\weak{\overset{\ast}{\rightharpoonup}}%

\global\long\def\intQ{\int\int_{Q_{T}}}%

\global\long\def\intO{\int_{\mathbb{T}^{d}}}%

\global\long\def\QT{\overline{Q_{T}}}%

\noindent \tableofcontents{}

\section{Introduction\label{sec:Introduction}}

In this paper, we prove existence of classical solutions to a broad
class of first order mean field games systems (MFG for short) with
a local coupling, which includes standard MFG, MFG with congestion,
and mean field type control problems. For this purpose, we study the
MFG system:

\begin{equation}
\tag{EMFG}\begin{cases}
-u_{t}+H(x,D_{x}u,m)=0 & (x,t)\in Q_{T}=\mathbb{T}^{d}\times(0,T),\\[1pt]
m_{t}-\textrm{div}(B(x,D_{x}u,m))=0 & (x,t)\in Q_{T},\\[4pt]
m(0,x)=m_{0}(x),\;u(x,T)=g(x,m(x,T)) & x\in\mathbb{T}^{d},
\end{cases}\label{eq:emfg}
\end{equation}
where $-H(x,p,m):\mathbb{T}^{d}\times\mathbb{R}^{d}\times(0,\infty)\rightarrow\mathbb{R}$
and $g(x,m):\mathbb{T}^{d}\times(0,\infty)\rightarrow\mathbb{R}$
are strictly increasing in $m$, and $m_{0}:\mathbb{T}^{d}\rightarrow\mathbb{R}$
is a positive probability density. 

MFG were introduced by Lasry and Lions \cite{LasryLions,Lions}, and
at the same time, in a particular setting, by Caines, Huang, and Malhamé
\cite{caines}. They are non-cooperative differential games with infinitely
many players, in which the players find an optimal strategy by observing
the distribution of the others. 

The system (\ref{eq:emfg}) was introduced by Lions and Souganidis
in \cite{LionsSoug}, who coined the term extended MFG, to simultaneously
study several MFG type problems for which, in contrast to the case
of standard MFG, the vector field $B$ does not necessarily equal
$mD_{p}H$. It was shown in \cite{LionsSoug} that (\ref{eq:emfg})
has at most one classical solution if
\begin{equation}
-4H_{m}D_{p}B>(B_{m}-D_{p}H)\otimes(B_{m}-D_{p}H).\label{eq:ellipticity weak}
\end{equation}
In the case of standard MFG with a \emph{separated} Hamiltonian, $H\equiv H(x,p)-f(x,m)$,
(\ref{eq:ellipticity weak}) simply reduces to the standard convexity
assumption for $H(x,p)$ in the second variable and the monotonicity
of $f$. 

Before stating our results, we briefly describe some of the existing
work on the well-posedness and regularity of (\ref{eq:emfg}). For
the case of standard MFG with a separated Hamiltonian, there exists
a complete theory of weak solutions (developed by Cardaliaguet, Graber,
Porretta, and Tonon \cite{Cardalaguiet,CardaliaguetGraber,CardaliaguetGraberPorrettaTonon}
in the degenerate case $g_{m}\equiv0$, that is, when $g$ is independent
of $m$, and by the author \cite{Munoz} in the non-degenerate case
considered here). Moreover, it was shown by the author, in \cite{Munoz},
that the solutions are classical under the coercivity assumption 
\begin{equation}
\lim_{m\rightarrow0^{+}}H(x,p,m)=+\infty.\label{eq:coercivity}
\end{equation}
From the optimal control perspective, (\ref{eq:coercivity}) corresponds
to placing a very strong incentive for players to occupy low-density
regions, and this forces $m>0$. For first order MFG systems with
congestion, weak solutions were shown to exist by A. Porretta and
Y. Achdou \cite{porretta slides}, but classical solutions had not
been obtained so far. Second order MFG systems with congestion were
also studied in \cite{achdou porretta,EvangelistaFerreiraGomesNurbekyanVoskanyan,Gomes,Graber},
where weak solutions and short-time existence result in the smooth
setting have been obtained. Finally, for second order mean field type
control problems with congestion, weak solutions were obtained in
\cite{achdou lauriere}. To put our results in context, it is important
to observe that assumption (\ref{eq:coercivity}), and it was not
made in \cite{achdou porretta,porretta slides}. This assumption is
a significant restriction, and it is critical to our methods, as it
ensures the strict positivity of the density, and hence the classical
solutions.

In this paper, we follow the same methodology used in \cite{Munoz}.
Our contribution is stated below. It is a general result that yields
classical solutions to (\ref{eq:emfg}), assuming the strict form
of the uniqueness condition (\ref{eq:ellipticity weak}), as well
as growth assumptions on $H$ and $B$ which are modeled by Hamiltonians
of the type 
\begin{equation}
H=\psi(m)|p|^{\gamma},\;\;\gamma>1,\;\psi>0,\;\psi'\leq0.\label{eq:H growth}
\end{equation}
We remark that, in particular, we do not require the Hamiltonian to
be quadratic or separated, and one of the  applications of Theorem
\ref{thm:smoothsols} is the existence of classical solutions to MFG
systems with congestion. We refer to Section \ref{sec:Assumptions}
for the exact assumptions  \hyperref[eq:M1]{(M)}, \hyperref[eq:Hpp bd]{(H)},
\hyperref[eq:DpB]{(B)} , \hyperref[eq: gx control]{(G)}, and \hyperref[eq:lower coer]{(E)}.
\begin{thm}
\label{thm:smoothsols}Let $0<s<1$, and assume that \hyperref[eq:M1]{(M)},
\hyperref[eq:Hpp bd]{(H)}, \hyperref[eq:DpB]{(B)} , \hyperref[eq: gx control]{(G)},
and \hyperref[eq:lower coer]{(E)} hold. Then there exists a unique
classical solution $(u,m)\in C^{3,s}(\overline{Q_{T}})\times C^{2,s}(\overline{Q_{T}})$
to \textup{(\ref{eq:emfg})}.
\end{thm}

An important natural setting, covered by Theorem \ref{thm:smoothsols}
in full generality, is when the derivatives $H_{m}$ and $D_{x}H$
satisfy growth conditions that are compatible with (\ref{eq:H growth}),
which we write symbolically as
\begin{equation}
mH_{m}\sim H\sim\psi(m)|p|^{\gamma},\;\text{and}\label{eq:natural hm}
\end{equation}
\begin{equation}
\text{ }|D_{x}H|\lesssim\psi(m)|p|^{\gamma}.\label{eq:natural Hx}
\end{equation}
When $\lim_{m\rightarrow\infty}\psi(m)=0$, these correspond to MFG
systems with congestion, of which a typical example is
\begin{equation}
\begin{cases}
-u_{t}+\frac{|D_{x}u|^{2}}{2(m+c_{0})^{\alpha}}-V(x)=f(m), & u(x,T)=g(x,m(x,T))\\[1pt]
m_{t}-\textrm{div}(\frac{m}{(m+c_{0})^{\alpha}}D_{x}u)=0, & m(0,x)=m_{0}(x)\\[4pt]
\end{cases}\label{eq:mfg cong}
\end{equation}
where the conditions $0<\alpha<2$ and $c_{0}\geq0$ ensure that (\ref{eq:ellipticity weak})
and, hence, uniqueness, holds for (\ref{eq:mfg cong}) (see \cite{achdou porretta}).
As was mentioned above, for our results to apply, condition (\ref{eq:coercivity})
is essential, and in this example it amounts to requiring that $\lim_{m\rightarrow0}f(m)=-\infty$.
In particular, Theorem \ref{thm:smoothsols} does not apply in several
important examples such as when $f\equiv0$ or $f\equiv m^{k}$, $k>0$,
which illustrates the restrictive character of (\ref{eq:coercivity}),
in contrast to the results of \cite{achdou porretta,porretta slides}.

Now, Theorem \ref{thm:smoothsols} also allows for a more general
growth behavior than (\ref{eq:natural hm}), namely, for a constant
$0\leq\gamma_{1}\leq\gamma$,
\[
mH_{m}\sim\psi(m)|p|^{\gamma_{1}}.
\]
One reason for working under such generality is that, despite (\ref{eq:natural hm})
being the natural condition for fully general Hamiltonians $H(x,p,m)$
satisfying (\ref{eq:H growth}), it is not satisfied by the important
example of MFG systems with a separated Hamiltonian. Such systems
are nevertheless covered in Theorem \ref{thm:smoothsols}, by setting
$\gamma_{1}=0$.

There are, however, two assumptions that must be strengthened when
straying from (\ref{eq:natural hm}). The first is that, whereas in
the case $\gamma_{1}=\gamma$, our result allows for ``congestions''
in which $\lim_{m\rightarrow\infty}\psi(m)=0$, when $\gamma_{1}<\gamma$
we must instead assume this limit to be positive. The second assumption
is the control required for the $x$-dependence, and can be explained
by comparing Theorem \ref{thm:smoothsols} with \cite[Theorem 1.1]{Munoz}.
In the latter work, we obtained classical solutions in the case of
separated Hamiltonians, only requiring for $D_{x}H(x,p)$ the condition
\[
|D_{x}H(x,p)|\lesssim|p|^{\gamma-\epsilon},\;\;\;\epsilon>0,
\]
which allows the growth of $D_{x}H$ to be arbitrarily close to the
natural one (\ref{eq:natural Hx}). This was achieved by exploiting
in a crucial way the separated structure of the system. On the other
hand, in (\ref{eq:emfg}), no such structure is available, and therefore
treating fully coupled Hamiltonians with $0\leq\gamma_{1}<\gamma$
forces us to impose the stricter control 
\[
|D_{x}H|\lesssim\psi(m)|p|^{\gamma_{2}},
\]
where $\gamma_{2}$ must satisfy $\gamma_{2}<2\gamma_{1}-\gamma+2$.
In other words, in the absence of additional structural assumptions,
the more the growth $\gamma_{1}$ of $H_{m}$ deviates from its natural
value $\gamma$, the more we must restrict the growth $\gamma_{2}$
of the space oscillation $D_{x}H$. 

We will discuss now our methods of proof. The key insight that allows
us to obtain classical solutions to this first order system is the
observation of Lions that, due to the strict monotonicity of $H$
with respect to $m$, one can eliminate the variable $m$ and transform
(\ref{eq:emfg}) into a second order quasilinear equation in $u$
with an oblique, non-linear boundary condition,
\begin{equation}
\tag{Q}\begin{cases}
Qu:=-\tr(A(x,Du)D^{2}u)+b(x,Du)=0 & \text{in }Q_{T},\\[1pt]
Nu:=N(x,t,u,Du)=0 & \text{on }\partial Q_{T},
\end{cases}\label{eq:quasilinear}
\end{equation}
where $Du=(D_{x}u,u_{t})$ and, for $(x,z,p,s)\in\mathbb{T}^{d}\times\mathbb{R\times\mathbb{R}}^{d}\times\mathbb{R}$,
\begin{align}
\tag{Q1}A(x,p,s)=\left(\frac{B_{m}+D_{p}H}{2},-1\right)\otimes\left(\frac{B_{m}+D_{p}H}{2},-1\right)-\begin{pmatrix}\frac{B_{m}-D_{p}H}{2}\otimes\frac{B_{m}-D_{p}H}{2}+H_{m}D_{p}B & 0\\
0 & 0
\end{pmatrix},\label{eq:matrix-1}
\end{align}
\renewcommand*{\theHequation}{nota3.\theequation}
\begin{align}
\tag{Q2}b(x,p,s)= & -D_{x}H(x,p,H^{-1})\cdot B_{m}(x,p,H^{-1})+H_{m}(x,p,H^{-1})\text{div}_{x}B(x,p,H^{-1}),\label{eq:first order term}\\
\tag{N}N(x,0,z,p,s)= & -s+H(x,p,m_{0}(x)),\;N(x,T,z,p,s)=-g(x,H^{-1}(x,p,s))+z,\label{eq:boundary}
\end{align}
and the function $H^{-1}(x,p,s)$ is the inverse of $H$ with respect
to $m$, defined by
\[
H^{-1}(x,p,H(x,p,m))=m.
\]
An important observation that can be seen directly from the definition
of $A$ is that this equation is elliptic precisely when (\ref{eq:ellipticity weak})
holds. For this reason, it is to be expected that the methods of quasilinear
elliptic equations with oblique boundary conditions, which were successful
in obtaining classical solutions to standard MFG systems in \cite{Lions,Munoz},
may also be applied in this more general setting. This is in fact
the approach that we follow here. Namely, we obtain a priori estimates
for $||u||_{C^{0}(\QT)}$ and $||Du||_{C^{0}(\QT)}$, and conclude
the existence of smooth solutions from the classical $C^{1,\alpha}$
estimates for oblique derivative problems (see \cite{Lieberman}),
the Schauder theory for linear oblique problems (see \cite{GilbargTrudinger,Lieberman solvability}),
and the non-linear method of continuity (see \cite{GilbargTrudinger}).

Finally, we discuss some of the newer results that have been obtained
after the first version of this work, as well as possible future directions.
In \cite{Porretta}, A. Porretta showed that one may still obtain
classical solutions to standard MFG with a separated Hamiltonian,
when $\mathbb{T}^{d}$ is replaced by a region in $\mathbb{R}^{d},$
in the setting of the so-called planning problem, where the terminal
density is a prescribed function. On the other hand, in \cite{MimikosMunoz},
N. Mimikos and the author showed that, when $d=1$, the key coercivity
assumption (\ref{eq:coercivity}) may be removed, and classical solutions
are obtained both in the setting of (\ref{eq:emfg}) and the planning
problem. It was also shown that one may weaken the assumption that
$m_{0}$ be strictly bounded away from $0$, and still obtain instant
regularization for times $t>0$, despite the loss of ellipticity at
the initial time. However, it remains an open question whether the
results of \cite{MimikosMunoz} may be extended to dimensions greater
than $1$ and, even in the case $d=1$, whether one may allow $m_{0}$
to vanish in a set of positive measure.
\begin{rem}
We note here that there is some ambiguity with the term \textit{extended
mean field games}, because it is also used to refer to standard MFG
systems in which the Hamiltonian depends on the acceleration of the
players (see, for instance, \cite{Gomes-1}). This is setting is unrelated
to the one present in this work and in \cite{LionsSoug}.
\end{rem}

\subsection*{Notation}

Let $n,k\in\mathbb{N}.$ Given $x,y\in\mathbb{R}^{n},$ $x$ and $y$
will always be understood to be row vectors, and their scalar product
$xy^{T}$ will be denoted by $x\cdot y$. For $0<s<1$, $C^{k,s}(\QT$)
refers to the space of $k$ times differentiable real-valued functions
with $s$--Hölder continuous $k^{\text{th}}$ order derivatives.
If $u\in C^{1}(\QT)$, the notation $Du$ will always refer to the
full gradient in all variables, whereas $D_{x}u$ denotes the gradient
in the space variable only. We write $C=C(K_{1},K_{2},\ldots,K_{M})$
for a positive constant $C$ depending monotonically on the non-negative
quantities $K_{1},\ldots,K_{M}.$ 

\section{\label{sec:Assumptions}Assumptions}

In what follows, $C_{0},\gamma,\gamma_{1},\gamma_{2}$ are fixed constants
satisfying 
\begin{equation}
C_{0}>0,\;\;\gamma>1,\;\;\gamma_{1}\geq0,\;\;\gamma_{2}\leq\gamma_{1}\leq\gamma,\;\;\text{and }\;\;\gamma_{2}<2\gamma_{1}-\gamma+2.\label{eq:condition exponents}
\end{equation}
\renewcommand*{\theHequation}{notag5.\theequation}The continuous
functions $\overline{C},\;\psi:(0,\infty)\rightarrow(0,\infty)$ are
also fixed, with $\psi$ being non-increasing. If $\gamma_{1}<\gamma$,
we further require that
\begin{equation}
\lim_{m\rightarrow\infty}\psi(m)>0.\label{eq:psi positive}
\end{equation}
We note that the case $\gamma_{1}=0$, $\psi\equiv1$, corresponds
to a standard MFG system with a separated Hamiltonian, and the case
\[
\gamma_{1}=\gamma,\;\;\psi(m)\equiv\frac{1}{(m+c_{0})^{\alpha}}
\]
corresponds to a MFG system with congestion. 
\begin{itemize}
\item[(M)] (Assumptions on $m_{0}$) The initial density $m_{0}$ satisfies
\begin{equation}
\tag{M1}m_{0}\in C^{4}(\mathbb{T}^{d}),\;m_{0}>0,\text{ and }\int_{\mathbb{T}^{d}}m_{0}=1.\label{eq:M1}
\end{equation}
\renewcommand*{\theHequation}{notag5.\theequation}
\item[(H)]  (Assumptions on $H$) The function $H:\mathbb{T}^{d}\times\mathbb{R}^{d}\times(0,\infty)\rightarrow\mathbb{R}$
is four times continuously differentiable and satisfies $H_{m}<0$.
Moreover, for $(x,p,m)\in\mathbb{T}^{d}\times\mathbb{R}^{d}\times(0,\infty)$,
\begin{equation}
\tag{H1}\frac{1}{C_{0}}\psi(m)(1+|p|){}^{\gamma-2}I\leq D_{pp}^{2}H\leq C_{0}\psi(m)(1+|p|)^{\gamma-2}I,\label{eq:Hpp bd}
\end{equation}
\renewcommand*{\theHequation}{notag6.\theequation}
\begin{equation}
\tag{H2}|D_{p}H|\leq C_{0}\psi(m)(1+|p|)^{\gamma-1},\;\;\;D_{p}H\cdot p\geq\left(1+\frac{1}{C_{0}}\right)H-\overline{C}(m),\label{eq:H qg}
\end{equation}
\renewcommand*{\theHequation}{notag7.\theequation}
\begin{equation}
\tag{HM1}\frac{1}{C_{0}}\psi(m)|p|^{\gamma_{1}}\leq-mH_{m}\leq C_{0}\psi(m)|p|^{\gamma_{1}}+\overline{C}(m),\label{eq:Hm growth}
\end{equation}
\renewcommand*{\theHequation}{notag7.\theequation}
\begin{equation}
\tag{HM2}|mH_{mm}|\leq-C_{0}H_{m},\;\;\;|p||D_{p}H_{m}|\leq C_{0}\psi(m)(1+|p|)^{\gamma_{1}},\label{eq:Hm bd}
\end{equation}
\renewcommand*{\theHequation}{notag8.\theequation}
\end{itemize}
\begin{equation}
\tag{HX1}|D_{x}H|,\;|D_{xx}^{2}H|\leq C_{0}\psi(m)(1+|p|)^{\gamma_{2}},\;\;|D_{xp}^{2}H|\leq C_{0}\psi(m)(1+|p|)^{\gamma_{2}-1},\label{eq:HX HXX HXP bd}
\end{equation}
\renewcommand*{\theHequation}{notag9.\theequation}
\begin{equation}
\tag{HX2}m|D_{x}H_{m}|\leq C_{0}\psi(m)(1+|p|)^{\gamma_{2}},\label{eq:mHmx bd}
\end{equation}
\renewcommand*{\theHequation}{notag10.\theequation}
\begin{equation}
\tag{HX3}|D_{x}H(x,0,m)|\leq C_{0}.\label{eq:DxH(0) bound}
\end{equation}
\renewcommand*{\theHequation}{notag11.\theequation}
\begin{itemize}
\item[(B)] (Assumptions on $B$) The function $B:\mathbb{T}^{d}\times\mathbb{R}^{d}\times[0,\infty)\rightarrow\mathbb{R}^{d}$
is four times continuously differentiable, $B(\cdot,\cdot,0)\equiv0$
and, mirroring the assumptions on $H$, $B$ satisfies, for $(x,p,m)\in\mathbb{T}^{d}\times\mathbb{R}^{d}\times(0,\infty)$,
\begin{equation}
\tag{B1}\frac{1}{C_{0}}m\psi(m)|p|{}^{\gamma-2}I\leq D_{p}B\leq C_{0}m\psi(m)(1+|p|)^{\gamma-2}I,\label{eq:DpB}
\end{equation}
\renewcommand*{\theHequation}{notag12.\theequation}
\begin{equation}
\tag{B2}|B_{m}|\leq C_{0}\psi(m)(1+|p|)^{\gamma_{1}-1},\;\;\;|p||D_{p}B_{m}|\leq C_{0}\psi(m)(1+|p|)^{\gamma_{1}-2},\;\;\;|D_{pp}^{2}B|\leq C_{0}m\psi(m)(1+|p|)^{\gamma-3},\label{eq: Bm bd, DpBm bd,  DppB bd}
\end{equation}
\renewcommand*{\theHequation}{notag13.\theequation}
\begin{equation}
\tag{BM}(1+|p|)|B_{mm}|\leq-C_{0}H_{m},\label{eq:Bmm bd}
\end{equation}
\renewcommand*{\theHequation}{notag14.\theequation}
\begin{equation}
\tag{BX1}|D_{x}B|,|D_{xx}^{2}B|\leq mC_{0}\psi(m)(1+|p|)^{\gamma_{2}-1},\;\;|D_{x}B_{m}|\leq C_{0}\psi(m)(1+|p|)^{\gamma_{2}-1},\label{eq:Bx Bxx Bxm bd}
\end{equation}
\renewcommand*{\theHequation}{notag15.\theequation}
\begin{equation}
\tag{BX2}|D_{xp}^{2}B|\leq C_{0}m\psi(m)(1+|p|)^{\gamma_{2}-2},\label{eq:Bxp bd}
\end{equation}
\renewcommand*{\theHequation}{notag16.\theequation}
\begin{equation}
\tag{BX3}|D_{x}B(x,0,m)|\leq C_{0}m.\label{eq:DxB(0) bound}
\end{equation}
\renewcommand*{\theHequation}{notag17.\theequation}
\item[(G)] (Assumptions on $g$) The function $g:\mathbb{T}^{d}\times(0,\infty)\rightarrow\mathbb{R}$
is four times continuously differentiable and satisfies $g_{m}>0$.
Furthermore, for each $x\in\mathbb{T}^{d}$,
\begin{align}
\tag{GX} & \lim_{m\rightarrow\infty}g(x,m)=\sup_{\mathbb{T}^{d}\times[0,\infty)}g,\text{ and }\lim_{m\rightarrow0^{+}}g(x,m)=\inf_{\mathbb{T}^{d}\times[0,\infty)}g.\label{eq: gx control}
\end{align}
\renewcommand*{\theHequation}{notag18.\theequation}
\item[(E)] (Strict ellipticity of the system) The functions $H$ and $B$ satisfy
the conditions
\begin{align}
\tag{E1}\lim_{m\rightarrow0^{+}}H(x,p,m)= & +\infty\text{\, uniformly in }(x,p)\in\mathbb{T}^{d}\times\mathbb{R}^{d},\label{eq:lower coer}
\end{align}
\renewcommand*{\theHequation}{notag19.\theequation}
\begin{equation}
\tag{E2}\lim_{m\rightarrow\infty}H(x,p,m)-C_{0}\psi(m)|p|^{\gamma}=-\infty\text{ \,uniformly in }(x,p)\in\mathbb{T}^{d}\times\mathbb{R}^{d},\label{eq:upper coer}
\end{equation}
\renewcommand*{\theHequation}{notag20.\theequation}
\begin{equation}
\tag{E3}-4H_{m}D_{p}B\geq\left(1+\frac{1}{C_{0}}\right)(B_{m}-D_{p}H)\otimes(B_{m}-D_{p}H).\label{eq:ellipticity}
\end{equation}
\renewcommand*{\theHequation}{notag21.\theequation}
\end{itemize}
\begin{rem}
In view of (\ref{eq:Hpp bd}), up to increasing the values of $C_{0}$
and $\overline{C}$, we may assume, with no loss of generality, that,
for $(x,p,m)\in\mathbb{T}^{d}\times\mathbb{R}^{d}\times(0,\infty)$,

\begin{equation}
\frac{\psi(m)}{C_{0}}|p|^{\gamma}-\overline{C}(m)\leq H(x,p,m)\leq\psi(m)C_{0}|p|^{\gamma}+\overline{C}(m).\label{eq:H lower and upper bounds}
\end{equation}
\renewcommand*{\theHequation}{notag22.\theequation}

\begin{equation}
\max_{[\min m_{0},\max m_{0}]}|H(x,0,\cdot)|\leq C_{0},\;\;\;\max_{[\min m_{0},\max m_{0}]}|B(x,0,\cdot)|\leq C_{0}.\label{eq:}
\end{equation}
Moreover, in the case that (\ref{eq:psi positive}) holds, we may
also write

\begin{equation}
\psi(m)\geq\frac{1}{C_{0}}.\label{eq:-1}
\end{equation}
We also note that there is certainly room for weaking the regularity
assumptions on the data, at the expense of further technical complications.
We refer to \cite{MimikosMunoz} for an illustration of this.
\end{rem}

\section{A priori estimates and classical solutions\label{sec:Classical-solutions}}

\subsection{Derivation of the quasilinear equation}

We begin by briefly showing the equivalence between the first-order
system (\ref{eq:emfg}) and the elliptic equation (\ref{eq:quasilinear}),
since the latter will be our main object of analysis in the following
sections.
\begin{prop}
\label{prop:quasil}Let $(u,m)\in C^{2}(\overline{Q_{T}})\times C^{1}(\overline{Q_{T}})$.
Then $(u,m)$ is a solution to (\ref{eq:emfg}) if and only if $u$
is a solution to (\ref{eq:quasilinear}), and $m$ is given by
\begin{equation}
m=H^{-1}(x,D_{x}u,u_{t}).\label{eq:eliminate m}
\end{equation}
\end{prop}

\begin{proof}
The Hamilton-Jacobi equation

\[
-u_{t}+H(x,D_{x}u,m)=0
\]
may be rewritten as (\ref{eq:eliminate m}). We thus need to show
that, after substituting (\ref{eq:eliminate m}) in the continuity
equation

\[
m_{t}-\text{div}(B(x,D_{x}u,m))=0,
\]
one obtains (\ref{eq:quasilinear}). Indeed, the substitution yields

\begin{multline*}
0=\frac{1}{H_{m}}(u_{tt}-D_{p}H\cdot D_{x}u_{t})-\text{div}_{x}B-\tr(D_{p}BD_{xx}^{2}u)-B_{m}\cdot\text{div}_{x}(H^{-1})\\
=\frac{1}{H_{m}}(u_{tt}-D_{p}H\cdot D_{x}u_{t})-\text{div}_{x}B-\tr(D_{p}BD_{xx}^{2}u)-\frac{1}{H_{m}}B_{m}\cdot(-D_{x}H+D_{p}HD_{xx}^{2}u+D_{x}u_{t}),
\end{multline*}
that is,

\begin{equation}
R+b(x,Du)=0,\label{eq:-3}
\end{equation}
where the first order term $b(x,Du)$ is given by (\ref{eq:first order term}),
and the second order term $R$ is

\[
R=-u_{tt}+(B_{m}+D_{p}H)\cdot D_{x}u_{t}-B_{m}D_{xx}^{2}u\cdot D_{p}H+H_{m}\tr(D_{p}BD_{xx}^{2}u).
\]
Now, $R$ may be rewritten as

\begin{multline}
R=-u_{tt}+2\frac{B_{m}+D_{p}H}{2}\cdot D_{x}u_{t}-\frac{B_{m}+D_{p}H}{2}D_{xx}^{2}u\cdot\frac{B_{m}+D_{p}H}{2}\\
+\frac{B_{m}-D_{p}H}{2}D_{xx}^{2}u\cdot\frac{B_{m}-D_{p}H}{2}+H_{m}\tr(D_{p}BD_{xx}^{2}u)=-\left(\frac{B_{m}+D_{p}H}{2},-1\right)D^{2}u\cdot\left(\frac{B_{m}+D_{p}H}{2},-1\right)\\
+\frac{B_{m}-D_{p}H}{2}D_{xx}^{2}u\cdot\frac{B_{m}-D_{p}H}{2}+H_{m}\tr(D_{p}BD_{xx}^{2}u)=-\tr(AD^{2}u).\label{eq:-4}
\end{multline}
where $A$ is given by (\ref{eq:matrix-1}). Substituting (\ref{eq:-4})
in (\ref{eq:-3}) thus yields the desired elliptic equation. As for
the boundary conditions, we may rewrite the initial and terminal conditions
in (\ref{eq:emfg}) as

\[
H^{-1}(x,D_{x}u(x,0),u_{t}(x,0))=m_{0}(x),\,\,\,g(x,H^{-1}(x,D_{x}u(x,T),u_{t}(x,T)))=u(x,T),
\]
that is,

\[
N(x,t,u,D_{x}u,u_{t})=0,\,\,\,(x,t)\in\mathbb{T}^{d}\times\{0,T\},
\]
where $N$ is given by (\ref{eq:boundary}).
\end{proof}
In view of Proposition \ref{prop:quasil}, (\ref{eq:emfg}) and (\ref{eq:quasilinear})
will be treated tacitly as the same problem throughout the rest of
the paper.

\subsection{\label{subsec:estimates u and m(T)}Estimates for the solution and
the terminal density}

In the first result of this section, Lemma \ref{lem:aprioriu}, we
will estimate the $L^{\infty}$ norms of $u$ and the terminal density
$m(\cdot,T)$, where $(u,m)$ is a solution to (\ref{eq:emfg}). In
order to provide an explicit form for the estimates of this section,
we consider the continuous, strictly increasing functions $f_{0},f_{1},g_{0},g_{1}:(0,\infty)\rightarrow\mathbb{R}$
defined by
\begin{align*}
f_{0}(m)=\min_{x\in\mathbb{T}^{d}}\left(-H(x,0,m)\right),\;\;f_{1}(m)=\max_{x\in\mathbb{T}^{d}}\left(-H(x,0,m)\right),\\
g_{0}(m)=\min_{\mathbb{T}^{d}}g(\cdot,m),\;\;g_{1}(m)=\max_{\mathbb{T}^{d}}g(\cdot,m),
\end{align*}
and the non-decreasing function $h:(0,\infty)\rightarrow[0,\infty)$
by

\begin{equation}
h(s)=\sup\{m>0:\sup_{(x,p)\in\mathbb{R}^{d}}H(x,p,m)-C_{0}|p|^{\gamma}\psi(m)\geq-s\},\label{eq:h defi}
\end{equation}
which is well-defined in view of (\ref{eq:upper coer}). 
\begin{lem}
\label{lem:aprioriu} There exists $C=C(C_{0})$ such that, for any
solution $(u,m)\in C^{2}(\overline{Q_{T}})\times C^{1}(\overline{Q_{T}})$
of \textup{(\ref{eq:emfg})}, and every $(x,t)\in\overline{Q_{T}},$
\begin{equation}
g_{0}f_{1}^{-1}(-C)-C(e^{CT}-e^{Ct})\leq u(x,t)\leq g_{1}f_{0}^{-1}(C)+C(e^{CT}-e^{Ct}),\text{ and}\label{aprioriu1}
\end{equation}
\begin{align}
0<g_{1}^{-1}g_{0}f_{1}^{-1}(-C)\leq & m(x,T)\leq g_{0}^{-1}g_{1}f_{0}^{-1}(C),\label{eq:a priori m(T) inequality}
\end{align}
\end{lem}

\begin{proof}
The proof of this statement is analogous to \cite[Lemma 3.1, Corollary 3.2]{Munoz}.
We modify the function $u$ in a way that ensures that its maximum
value is achieved at $t=T$. This will allow us to conclude by exploiting
the fact that the boundary condition (\ref{eq:boundary}) that holds
at the terminal time is of ``Robin type''. For this purpose, we
set $v(x,t)=u(x,t)+\zeta(t)$, where $\zeta(t)=M(e^{Mt}-e^{MT})$,
for a large parameter $M>1$. Conditions (\ref{eq:Hm bd}) and (\ref{eq:Bmm bd})
imply, respectively, the existence of a uniform Lipschitz bound for
the maps $w\mapsto H^{-1}(x,0,w)H_{m}(x,0,H^{-1}(x,0,w))$ and $w\mapsto B_{m}(x,0,H^{-1}(x,0,w))$.
Therefore, using (\ref{eq:}), we obtain
\[
|mH_{m}(x,0,m)|\leq C(1+|H(x,0,m)|)\text{ and }|B_{m}(x,0,m)|\leq C(1+|H(x,0,m)|).
\]
In view of this, (\ref{eq:quasilinear}), (\ref{eq:first order term}),
(\ref{eq:DxH(0) bound}), and (\ref{eq:DxB(0) bound}) yield that,
at any interior critical point $(x,t)$ of $v$,
\begin{multline*}
-\tr(A(x,Du)D^{2}v)=-\tr(A(x,Du)D^{2}u)-\zeta''(t)=-b(x,Du)-\zeta''(t)=D_{x}H(x,0,m)\cdot B_{m}(x,0,m)\\
-H_{m}(x,0,m)\text{div}_{x}B(x,0,m)-\zeta''(t)\leq C(1+|u_{t}|)-\zeta''(t)=C(1+\zeta'(t))-\zeta''(t)\leq C(1+M^{2}e^{Mt})-M^{3}e^{Mt}.
\end{multline*}
Thus, if $M>\max(1,2C)$, one has $-\tr(A(x,Du)D^{2}v)<0$ at all
interior critical points of $v$, and therefore $v$ must achieve
its maximum value on the boundary $\partial Q_{T}$. If the maximum
is achieved at $t=0$, then $D_{x}v=0,\;v_{t}\leq0$, and so
\[
M^{2}=\zeta'(0)\leq-u_{t}=-H(x,0,m_{0}(x))\leq C_{0}.
\]
Consequently, a sufficiently large value of $M$ forces the maximum
to be achieved at $\{t=T\}$. At such a point $(x,T)$, $D_{x}v=0,\;v_{t}\geq0$,
and, thus, since $\zeta(T)=0$,
\[
M^{2}e^{MT}=\zeta'(T)\geq-u_{t}=-H(x,0,m(x,T))\geq f_{0}(m(x,T))=f_{0}(g^{-1}(x,u(x,T))=f_{0}(g^{-1}(x,v(x,T)),
\]
which yields $v(x,T)\leq g_{1}f_{0}^{-1}(M^{2}e^{MT})$. Since $(x,T)$
is a maximum point of $v=u+\zeta$, this proves the upper bound in
(\ref{aprioriu1}), with the lower bound being obtained through the
same reasoning. 

The second inequality (\ref{eq:a priori m(T) inequality}) then follows
immediately by setting $t=T$ in (\ref{aprioriu1}), using the fact
that, by (\ref{eq: gx control}), the functions $g_{0}$ and $g_{1}$
have the same range.
\end{proof}

\subsection{An overview of the Bernstein argument}

To obtain the gradient estimate, we will make use of a classical method
due to S. Bernstein (see \cite{Bernstein}), for which we will need
to use the linearization of (\ref{eq:quasilinear}), namely
\begin{equation}
L_{u}(v)=-\tr(A(x,Du)D^{2}v)-D_{q}\tr(A(x,Du)D^{2}u)\cdot Dv+D_{q}b(x,Du)\cdot Dv,\label{eq:Linearization definition}
\end{equation}
where, for $(p,s)\in\mathbb{R}^{d}\times\mathbb{R}$, we denote $q=(p,s)$.
The idea behind this classical method is the following general principle
about elliptic equations: convex functions $\phi(u)$ and $\Phi(Du)$
of the solution and its gradient are subsolutions of the linearized
equation, up to an error that can often be controlled. More precisely,
one has
\begin{equation}
L_{u}(\phi(u))=-\phi''DuA\cdot Du+E_{1},\;\;\;L_{u}(\Phi(Du))=-\tr(D^{2}\Phi D^{2}uAD^{2}u)+E_{2},\label{eq:bernstein inf}
\end{equation}
where $E_{1}$ and $E_{2}$ are regarded as error terms to be estimated.
This observation can be exploited to bound $||D_{x}u||_{C^{0}(\QT)}$
as follows. Since $||u||_{C^{0}(\QT)}$ is already known to be bounded
a priori, the problem is equivalent to bounding $v=\phi(u)+\Phi(Du)$,
as long as $\Phi:\mathbb{R}^{d}\rightarrow\mathbb{R}$ is coercive.
At any interior maximum point $(x,t)$ of $v$, one thus has

\[
0\leq L_{u}(v)=-\phi''DuA\cdot Du-\tr(D^{2}\Phi D^{2}uAD^{2}u)+E_{1}+E_{2},
\]
that is,
\begin{equation}
\phi''DuA\cdot Du\leq-\tr(D^{2}\Phi D^{2}uAD^{2}u)+E_{1}+E_{2}.\label{eq:bernstein informal}
\end{equation}
Thus, up to adequately estimating the error $E_{1}+E_{2}$ in terms
of the other two dominant signed terms, (\ref{eq:bernstein informal})
leads naturally to a gradient bound. 

Now, we must note that the argument above applies only to interior
maxima, so the possibility of the maximum being achieved on $\partial Q_{T}$
must be accounted for separately. In the usual case of Dirichlet boundary
conditions, the bound would follow automatically since $u|_{\partial Q_{T}}$
would be an a priori given function, but since (\ref{eq:boundary})
defines an oblique boundary condition instead, an additional argument
must be made. One can proceed by linearizing the boundary operator
$N$ and repeating the Bernstein process for this first order operator
in place of $L_{u}$. Just like in the case of (\ref{eq:bernstein inf}),
the linearization is computed by differentiating both sides of the
boundary equation. Whereas the ellipticity of (\ref{eq:quasilinear})
is what allows $E_{1}+E_{2}$ in (\ref{eq:bernstein informal}) to
be estimated, the error at the boundary is instead controlled with
the dominant signed term $D_{p}H\cdot D_{x}u$ by virtue of the superlinear
growth (\ref{eq:H qg}) of $H$, the existing bounds on $m|_{\partial Q_{T}}$
and the non-degeneracy of the boundary condition. Indeed, bounds on
$m(\cdot,0)$ and $D_{x}m(\cdot,0)$ are available because $m_{0}$
is given a priori, and Lemma \ref{lem:aprioriu} provides bounds for
$m(\cdot,T)$, albeit not for $D_{x}m(\cdot,T)$. The error terms
that involve $D_{x}m(\cdot,T)$ have, however, a favorable sign thanks
to the ``Robin type'' nature of (\ref{eq:boundary}) at time $T$
that comes from the strict monotonicity of $g$.

\subsection{Estimates for the space-time gradient}

To carry out the strategy described above for the gradient estimate,
we will require explicit computations of the error terms $E_{1}$
and $E_{2}$ described in (\ref{eq:bernstein inf}), provided by the
following lemma. We remind the reader that $\cdot$ denotes the standard
dot product, and all vectors are taken to be rows. 
\begin{lem}
\label{lem:(Bernstein's-method)}Let $\Phi(p,s)\in C^{2}(\mathbb{R}^{d+1})$,
assume that $(u,m)\in C^{3}(\overline{Q_{T}})\times C^{2}(\overline{Q_{T}})$
solves \textup{(\ref{eq:emfg}),} and set $v(x,t)=\Phi(Du(x,t))$.
For each $q=(p,s)\in\mathbb{R}^{d+1}$, and for each $(x,t)\in Q_{T}$,
define 
\[
\zeta(p,s)=-s+D_{p}H(x,D_{x}u,m)\cdot p,\;\;Y^{+}=B_{m}+D_{p}H,\;\;Y^{-}=B_{m}-D_{p}H
\]
Then the following identities hold:
\begin{multline}
D_{q}\text{\emph{\ensuremath{\tr}}}(AD^{2}u)\cdot q=(-D_{x}u_{t}+\frac{1}{2}Y^{+}D_{xx}^{2}u)(D_{p}Y^{+}p^{T}-(Y_{m}^{+})^{T}H_{m}^{-1}\zeta)+\frac{1}{2}Y^{-}D_{xx}^{2}u(D_{p}Y^{-}p^{T}-(Y_{m}^{-})^{T}H_{m}^{-1}\zeta)\\
-(D_{p}H_{m}\cdot p-H_{mm}H_{m}^{-1}\zeta)\text{\emph{\ensuremath{\tr}}}(D_{p}BD_{xx}^{2}u)-H_{m}(D_{p}\text{\emph{\ensuremath{\tr}}}(D_{p}BD_{xx}^{2}u)\cdot p-\text{\emph{\ensuremath{\tr}}}(D_{p}B_{m}D_{xx}^{2}u)H_{m}^{-1}\zeta),\label{eq:D_q(A)}
\end{multline}
\begin{multline}
D_{q}b(x,Du)\cdot q=-B_{m}(D_{xp}^{2}Hp^{T}-D_{x}H_{m}^{T}H_{m}^{-1}\zeta)-D_{x}H(D_{p}B_{m}p^{T}-B_{mm}^{T}H_{m}^{-1}\zeta)\\
+H_{m}(D_{p}\textup{div}_{x}B\cdot p-\textup{div}_{x}B_{m}H_{m}^{-1}\zeta)+\textup{div}_{x}B(D_{p}H_{m}\cdot p-H_{mm}H_{m}^{-1}\zeta),\label{eq:Dqb}
\end{multline}
\begin{multline}
\text{\emph{\ensuremath{\tr}}}(A_{x_{i}}D^{2}u)=(-D_{x}u_{t}+\frac{1}{2}Y^{+}D_{xx}^{2}u)\cdot(Y_{x_{i}}^{+}-Y_{m}^{+}H_{m}^{-1}H_{x_{i}})+\frac{1}{2}Y^{-}D_{xx}^{2}u\cdot(Y_{x_{i}}^{-}-Y_{m}^{-}H_{m}^{-1}H_{x_{i}})\\
-(H_{x_{i}m}-H_{mm}H_{m}^{-1}H_{x_{i}})\text{\emph{\ensuremath{\tr}}}(D_{p}BD_{xx}^{2}u)-H_{m}(\text{\emph{\ensuremath{\tr}}}(D_{p}B_{x_{i}}D_{xx}^{2}u)-\text{\emph{\ensuremath{\tr}}}(D_{p}B_{m}D_{xx}^{2}u)H_{m}^{-1}H_{x_{i}}),\label{eq:Axi}
\end{multline}

\begin{multline}
D_{x}b(x,Du)\cdot p=-B_{m}(D_{xx}^{2}Hp^{T}-D_{x}H_{m}^{T}H_{m}^{-1}(D_{x}H\cdot p))-D_{x}H(D_{x}B_{m}p^{T}-B_{mm}H_{m}^{-1}(D_{x}H\cdot p))\\
+H_{m}(D_{x}\textup{div}_{x}(B)\cdot p-\textup{div}_{x}B_{m}H_{m}^{-1}(D_{x}H\cdot p))+\textup{div}_{x}B(D_{x}H_{m}\cdot p-H_{mm}H_{m}^{-1}(D_{x}H\cdot p)),\label{eq:Dxb}
\end{multline}
\begin{align}
L_{u}v= & -\text{\emph{\ensuremath{\tr}}}(D^{2}\Phi D^{2}uAD^{2}u)+\sum_{i=1}^{d}\emph{\ensuremath{\tr}}(A_{x_{i}}D^{2}u)\Phi_{p_{i}}-D_{p}\Phi\cdot D_{x}b.\label{eq:bernst}
\end{align}
\end{lem}

\begin{proof}
We derive (\ref{eq:D_q(A)}) differentiating the expressions (\ref{eq:matrix-1})
with respect to $q=(p,s)$. Indeed, (\ref{eq:matrix-1}) implies that

\begin{align*}
D_{q}\text{\emph{\ensuremath{\tr}}}(AD^{2}u)\cdot q= & D_{q}(\tr\left(\left(\frac{1}{4}(Y^{+}-Y^{-})\otimes(Y^{+}-Y^{-})-H_{m}D_{p}B\right)D_{xx}^{2}u\right)-Y^{+}D_{x}u_{t})\cdot(p,s)\\
= & D_{q}(\frac{1}{4}Y^{+}D_{xx}^{2}u\cdot Y^{+}+\frac{1}{4}Y^{-}D_{xx}^{2}u\cdot Y^{-}-H_{m}\tr(D_{p}BD_{xx}^{2}u)-Y^{+}D_{x}u_{t})\cdot(p,s)\\
= & \frac{1}{2}Y^{+}D_{xx}^{2}u(D_{p}Y^{+}p^{T}-(Y_{m}^{+})^{T}H_{m}^{-1}\zeta)+\frac{1}{2}Y^{-}D_{xx}^{2}u(D_{p}Y^{-}p^{T}-(Y_{m}^{-})^{T}H_{m}^{-1}\zeta)\\
 & -(D_{p}H_{m}\cdot p-H_{mm}H_{m}^{-1}\zeta)\tr(D_{p}BD_{xx}^{2}u)-H_{m}(D_{p}\tr(D_{p}BD_{xx}^{2}u)\cdot p\\
 & -\tr(D_{p}B_{m}D_{xx}^{2}u)H_{mm}H_{m}^{-1}\zeta)-D_{x}u_{t}(D_{p}Y^{+}p^{T}-Y_{m}^{+}H_{m}^{-1}\zeta)\\
= & (-D_{x}u_{t}+\frac{1}{2}Y^{+}D_{xx}^{2}u)(D_{p}Y^{+}p^{T}-(Y_{m}^{+})^{T}H_{m}^{-1}\zeta)+\frac{1}{2}Y^{-}D_{xx}^{2}u(D_{p}Y^{-}p^{T}-(Y_{m}^{-})^{T}H_{m}^{-1}\zeta)\\
 & -(D_{p}H_{m}\cdot p-H_{mm}H_{m}^{-1}\zeta)\tr(D_{p}BD_{xx}^{2}u)-H_{m}(D_{p}\tr(D_{p}BD_{xx}^{2}u)\cdot p-\tr(D_{p}B_{m}D_{xx}^{2}u)H_{m}^{-1}\zeta).
\end{align*}
Similarly, (\ref{eq:Dqb}) follows by differentiating (\ref{eq:first order term})
with respect to $q$, and (\ref{eq:Axi}) and (\ref{eq:Dxb}) result
from differentiating (\ref{eq:matrix-1}) and (\ref{eq:first order term})
with respect to the space variables. Finally, (\ref{eq:bernst}) is
obtained by applying $D\Phi\cdot D$ to both sides of (\ref{eq:quasilinear})
(see, for instance, \cite[Lemma 3.4]{Munoz}).
\end{proof}
We can now obtain the a priori gradient bound in terms of bounds for
the solution $u$ and the terminal density $m(\cdot,T)$, which were
already obtained in Section \ref{subsec:estimates u and m(T)}. 
\begin{lem}
\label{lem: gradient a priori bound} Let $(u,m)\in C^{3}(\overline{Q_{T}})\times C^{2}(\overline{Q_{T}})$
be a solution to \textup{(\ref{eq:emfg})}, and set 
\begin{align}
M=||m||_{C^{0}(\partial Q_{T})}+||m^{-1}||_{C^{0}(\partial Q_{T})}.\label{eq:beta_K}
\end{align}
There exist constants $C,C_{1}>0$, with
\begin{gather*}
C=C(C_{1},||(\psi\circ h)^{-1}||_{C^{0}(0,C_{1}]}),\;\;C_{1}=C_{1}(C_{0},T,T^{-1},||u||_{C^{0}(\QT)},M,||\overline{C}||_{C^{0}[\frac{1}{M},M]},||\psi||_{C^{0}[\frac{1}{M},M]},\\
||\psi^{-1}||_{C^{0}[\frac{1}{M},M]},||D_{x}g||_{C^{0}(\mathbb{T}^{d}\times[\frac{1}{M},M])},(2\gamma_{1}-\gamma+2-\gamma_{2})^{-1})
\end{gather*}
such that
\[
||Du||_{C^{0}(\QT)}\leq C.
\]
\end{lem}

\begin{proof}
We will consider first, for the sake of clarity, the natural case
where $\gamma_{1}=\gamma_{2}=\gamma$. $C$ will denote a constant
that is allowed to increase from line to line. First, we verify that
it is sufficient to bound the space gradient. Indeed, setting $\Phi(p,s)=s$
in Lemma \ref{lem:(Bernstein's-method)} yields
\[
L_{u}(u_{t})=L_{u}(\Phi(Du))=0,
\]
and, thus, in view of the maximum principle and (\ref{eq:H lower and upper bounds}),

\begin{equation}
-C\leq u_{t}\leq C||D_{x}u||_{\QT}^{\gamma}+C.\label{eq:ut bound}
\end{equation}
We note that, in (\ref{eq:ut bound}), the constant $C$ already depends
on the upper and lower bounds for $m$ on $\partial Q_{T}$. Next,
we will estimate $||D_{x}u||_{\QT}$ through the Bernstein method.
Let
\[
T_{u}v=-v_{t}+D_{p}H(x,D_{x}u,m)D_{x}v,\;\;\widetilde{u}=u+||u||_{C^{0}(\QT)}+1-\frac{2(||u||_{C^{0}(\QT)}+1)}{T}(T-t),
\]
and note that the function $\utilde$ has been constructed to satisfy
\begin{equation}
|\utilde|\leq C,\quad\widetilde{u}(\cdot,0)\leq-1,\;\widetilde{u}(\cdot,T)\geq1.\label{signs utilde}
\end{equation}
Setting
\begin{equation}
k=||D_{x}u||_{\QT}^{3/2},\;\;v(x,t)=\frac{k}{2}\widetilde{u}^{2}+\frac{1}{2}|D_{x}u|^{2},\label{eq: k defi}
\end{equation}
we observe that the quantities $||D_{x}u||_{\QT}^{2}$ and $||v||_{\QT}$
are comparable up to constants, so it is therefore sufficient to obtain
a bound for the latter.

Let $(x_{0},t_{0})\in\overline{Q_{T}}$ be a point where $v$ achieves
its maximum value, and set $p=D_{x}u(x_{0},t_{0})$. We assume with
no loss of generality that 
\begin{equation}
|p|\geq1\,\,\text{ and }\,\,||D_{x}u||_{\QT}^{1/2}\geq2||\utilde||_{\QT}^{2}.\label{eq:p>=00003D1}
\end{equation}
The latter condition ensures that 
\begin{equation}
\frac{1}{2}|p|^{2}\geq\frac{1}{2}||D_{x}u||_{\QT}^{2}-\frac{k}{2}||\utilde||_{\QT}^{2}\geq\frac{1}{4}||D_{x}u||_{\QT}^{2}.\label{eq:|p|^2 geq ||Dxu||^2}
\end{equation}
Since the maximum may be achieved at the boundary of $Q_{T}$, we
must distinguish three cases.

\textbf{Case 1: }$t_{0}=T$. Then $D_{x}v=0$, $v_{t}\geq0$. Therefore,
in view of (\ref{eq:H lower and upper bounds}), (\ref{signs utilde}),
(\ref{eq:Hm growth}), (\ref{eq:H qg}), (\ref{eq:|p|^2 geq ||Dxu||^2}),
and the current assumption that $\gamma_{1}=\gamma_{2}=\gamma,$
\begin{multline*}
0\geq T_{u}v=T_{u}\left(\frac{1}{2}|D_{x}u|^{2}\right)+k\utilde(-\widetilde{u}_{t}+D_{p}H\cdot D_{x}u)\\
=-\frac{H_{m}}{g_{m}}(|p|^{2}-D_{x}g\cdot p)-D_{x}H\cdot p+k\utilde(-u_{t}+D_{p}H\cdot p-C)\geq\frac{\psi(m(T))}{C_{0}m(T)g_{m}}|p|^{\gamma+2}\\
-\left(C_{0}\frac{\psi(m(T))}{m(T)g_{m}}|p|^{\gamma}+\overline{C}(m(T))\right)|D_{x}g||p|-C\psi(m(T))|p|^{\gamma+1}+k\utilde\left(\frac{1}{C}\psi(m(T))|p|^{\gamma}-C\right)\\
\geq\frac{1}{C}|p|^{\gamma+3/2}+\frac{\psi(m(T))}{C_{0}m(T)g_{m}}|p|^{\gamma+2}-C(1+|p|^{\gamma+1}).
\end{multline*}
Thus, since the second term is non-negative, we obtain
\[
\frac{1}{C}|p|^{\gamma+3/2}\leq C(1+|p|^{\gamma+1}),
\]
which yields 
\[
|D_{x}u|\leq C.
\]

\textbf{Case 2: }$t_{0}=0$. Similarly to the first case, we obtain
$D_{x}v=0$, $v_{t}\leq0$, and so
\begin{align*}
0\leq T_{u}v= & T_{u}\left(\frac{1}{2}|D_{x}u|^{2}\right)+k\utilde(-\widetilde{u}_{t}+D_{p}H\cdot D_{x}u)\\
= & -H_{m}D_{x}m_{0}(x)\cdot p-D_{x}H\cdot p+k\utilde(-u_{t}+D_{p}H\cdot p-C)\\
\leq & Cm_{0}^{-1}(|p|^{\gamma}\psi(m_{0})+\overline{C}(m_{0}))|p|+C\psi(m_{0})|p|^{\gamma+1}+k\utilde(\frac{1}{C_{0}}\psi(m_{0})|p|^{\gamma}-C)\\
\leq & -\frac{1}{C}\psi(m_{0})|p|^{\gamma+3/2}+C(1+|p|^{\gamma+1}),
\end{align*}
and, once more, we conclude that
\[
|D_{x}u|\leq C.
\]

\textbf{Case 3: }$0<t_{0}<T$. Then $Dv=0$, $D^{2}v\leq0$, which
yields
\begin{equation}
0\leq L_{u}v.\label{eq:linearization ineq}
\end{equation}
By direct computation, we see from (\ref{eq:quasilinear}) that

\[
L_{u}\left(\frac{1}{2}\utilde^{2}\right)=-D\utilde AD\utilde+\utilde L_{u}(\utilde)=-D\utilde AD\utilde-\utilde D_{q}\tr(AD^{2}u)\cdot D\utilde+\utilde D_{q}b\cdot D\utilde-\utilde b,
\]
whereas letting $\Phi(p,s)=\frac{1}{2}|p|^{2}$ in Lemma \ref{lem:(Bernstein's-method)},

\[
L_{u}\left(\frac{1}{2}|D_{x}u|^{2}\right)=-\sum_{i=1}^{d}Du_{x_{i}}A\cdot Du_{x_{i}}+\sum_{i=1}^{d}\tr(A_{x_{i}}D^{2}u)u_{x_{i}}-D_{x}b\cdot p,
\]
and thus
\begin{align}
L_{u}(v) & =-kD\utilde A\cdot D\utilde-\sum_{i=1}^{d}Du_{x_{i}}A\cdot Du_{x_{i}}+E,\label{eq:lineariz(v)}
\end{align}
where $E$ is the error term, computed as follows. Setting $\Lambda=D_{x}+\utilde kD_{p}$
and using Lemma \ref{lem:(Bernstein's-method)} once more, we have
$E=E_{1}+E_{2},$ with

\begin{multline}
E_{1}=(-D_{x}u_{t}+\frac{1}{2}Y^{+}D_{xx}^{2}u)\Lambda Y^{+}\cdot p+\frac{1}{2}Y^{-}D_{xx}^{2}u\Lambda Y^{-}\cdot p-\tr(D_{p}BD_{xx}^{2}u)\Lambda H_{m}\cdot p\\
-H_{mm}H_{m}^{-1}(D_{x}H\cdot p+\utilde k\zeta)\tr(D_{p}BD_{xx}^{2}u)-H_{m}(\tr(\Lambda D_{p}BD_{xx}^{2}u)\cdot p\\
-\tr(D_{p}B_{m}D_{xx}^{2}u)H_{m}^{-1}(D_{x}H\cdot p+\utilde k\zeta)),\label{eq:E1 terms}
\end{multline}
\begin{multline}
E_{2}=-B_{m}\cdot(\Lambda D_{x}Hp^{T}-D_{x}H_{m}H_{m}^{-1}(D_{x}H\cdot p+\utilde k\zeta))-D_{x}H\cdot(\Lambda B_{m}p^{T}-B_{mm}H_{m}^{-1}(D_{x}H\cdot p+\utilde k\zeta))\\
+H_{m}(\Lambda\text{div}_{x}Bp^{T}-\text{div}_{x}B_{m}H_{m}^{-1}(D_{x}H\cdot p+\utilde k\zeta))+\textup{div}_{x}B(\Lambda H_{m}\cdot p-H_{mm}H_{m}^{-1}(D_{x}H\cdot p+\utilde k\zeta))-k\utilde b.\label{eq:E2 terms}
\end{multline}
Before estimating the $E_{i}$, compute lower bounds for the dominant
signed terms in (\ref{eq:lineariz(v)}), in the following way. Setting
$r=(1+\frac{1}{2C_{0}})^{-1}$, and using (\ref{eq:matrix-1}), we
may write 

\begin{multline*}
DuA\cdot Du=|-u_{t}+\frac{1}{2}Y^{+}\cdot p|^{2}-|\frac{1}{2}Y^{-}\cdot p|^{2}-H_{m}pD_{p}B\cdot p=|-u_{t}+\frac{1}{2}Y^{+}\cdot p|^{2}\\
-|\frac{1}{2}Y^{-}\cdot p|^{2}-rH_{m}pD_{p}B\cdot p-(1-r)H_{m}pD_{p}B\cdot p.
\end{multline*}
Now, in view of (\ref{eq:ellipticity}), observing that $r(1+\frac{1}{C_{0}})>1$
and $r<1$, we obtain

\begin{multline}
DuA\cdot Du\geq|-u_{t}+\frac{1}{2}Y^{+}\cdot p|^{2}-|\frac{1}{2}Y^{-}\cdot p|^{2}+r(1+\frac{1}{C_{0}})|\frac{1}{2}Y^{-}\cdot p|^{2}-(1-r)H_{m}pD_{p}B\cdot p\\
=|-u_{t}+\frac{1}{2}Y^{+}\cdot p|^{2}+(r(1+\frac{1}{C_{0}})-1)|\frac{1}{2}Y^{-}\cdot p|^{2}-(1-r)H_{m}pD_{p}B\cdot p\\
\geq|-u_{t}+\frac{1}{2}Y^{+}\cdot p|^{2}+\frac{1}{C}|\frac{1}{2}Y^{-}\cdot p|^{2}-\frac{1}{C}H_{m}pD_{p}B\cdot p.\label{eq:signed 1}
\end{multline}
Similarly, (\ref{eq:matrix-1}) and (\ref{eq:ellipticity}) yield

\begin{multline}
\sum_{i=1}^{d}Du_{x_{i}}A\cdot Du_{x_{i}}=|-D_{x}u_{t}+\frac{1}{2}Y^{+}D_{xx}^{2}u|^{2}-|\frac{1}{2}Y^{-}D_{xx}^{2}u|^{2}-H_{m}\tr D_{xx}^{2}uD_{p}B\cdot D_{xx}^{2}u\\
\geq|-D_{x}u_{t}+\frac{1}{2}Y^{+}D_{xx}^{2}u|^{2}+\frac{1}{C}|\frac{1}{2}Y^{-}D_{xx}^{2}u|^{2}-\frac{1}{C}H_{m}\tr(D_{xx}^{2}uD_{p}B\cdot D_{xx}^{2}u).\label{signed 2}
\end{multline}

Therefore, (\ref{eq:matrix-1}) and (\ref{eq:ellipticity}) yield
\begin{align}
D\utilde A\cdot D\utilde & =|-\utilde_{t}+\frac{1}{2}Y^{+}\cdot p|^{2}-|\frac{1}{2}Y^{-}\cdot p|^{2}-H_{m}pD_{p}B\cdot p\geq\frac{1}{2}DuA\cdot Du-C,\label{eq:signed 2''}
\end{align}
and, on the other hand, since $\frac{1}{2}(Y^{+}+Y^{-})=D_{p}H$,
\begin{equation}
|\zeta|^{2}=|-\tilde{u}_{t}+D_{p}H\cdot D_{x}\utilde|^{2}\leq2(|-\utilde_{t}+\frac{1}{2}Y^{+}\cdot D_{x}\utilde|^{2}+|\frac{1}{2}Y^{-}\cdot D_{x}\utilde|^{2})\leq CD\utilde A\cdot D\utilde.\label{eq:signed 2'}
\end{equation}

Applying Young's inequality and (\ref{signed 2}) in (\ref{eq:E1 terms}),
we obtain
\begin{multline}
|E_{1}|\leq\frac{1}{2}\sum_{i=1}^{d}Du_{x_{i}}A\cdot Du_{x_{i}}+C[|\Lambda Y^{+}|^{2}|p|^{2}+|\Lambda Y^{-}|^{2}|p|^{2}+|H_{m}|^{-1}|D_{p}B||\Lambda H_{m}|^{2}|p|^{2}\\
+|H_{mm}|^{2}|H_{m}|^{-3}(|D_{x}H|^{2}|p|^{2}+k^{2}\zeta^{2})|D_{p}B|+|H_{m}||\Lambda D_{p}B|^{2}|D_{p}B|^{-1}|p|^{2}\\
+|D_{p}B_{m}|^{2}|D_{p}B|^{-1}|H_{m}|^{-1}(|D_{x}H|^{2}|p|^{2}+k^{2}\zeta^{2})].\label{eq:E1 terms'}
\end{multline}
Now, the terms in (\ref{eq:E1 terms'}) may all be estimated with
the help of the growth assumptions \hyperref[eq:Hpp bd]{(H)} and
\hyperref[eq:DpB]{(B)}. Indeed, in view of (\ref{eq:Hpp bd}), (\ref{eq:HX HXX HXP bd}),
(\ref{eq:mHmx bd}), (\ref{eq:Bx Bxx Bxm bd}), (\ref{eq: Bm bd, DpBm bd,  DppB bd}),
and (\ref{eq:p>=00003D1}), we estimate
\begin{equation}
|\Lambda Y^{+}|^{2},|\Lambda Y^{-}|^{2}\leq C(|D_{xp}^{2}H|^{2}+|D_{x}B_{m}|^{2}+k^{2}|D_{pp}^{2}H|^{2}+k^{2}|D_{p}B_{m}|^{2})\leq C\psi(m)^{2}(|p|^{2\gamma-2}+|p|^{2\gamma-1}),\label{eq:Lambda1}
\end{equation}
\begin{equation}
|\Lambda H_{m}|^{2}\leq C|D_{x}H_{m}|^{2}+Ck^{2}|D_{p}H_{m}|^{2}\leq C\psi(m)^{2}m^{-2}(|p|^{2\gamma}+|p|^{2\gamma+1}),\label{Lambda2}
\end{equation}
\begin{equation}
|\Lambda D_{p}B|^{2}\leq C|D_{xp}^{2}B|^{2}+Ck^{2}|D_{pp}^{2}B|^{2}\leq C\psi(m)^{2}m^{2}(|p|^{2\gamma-4}+|p|^{2\gamma-3}).\label{Lambda 3}
\end{equation}
Thus, using (\ref{eq:Hm growth}), (\ref{eq:Hm bd}), (\ref{eq:DpB}),
(\ref{eq: Bm bd, DpBm bd,  DppB bd}), (\ref{eq:HX HXX HXP bd}),
(\ref{eq:Lambda1}), (\ref{Lambda2}), and (\ref{Lambda 3}) in (\ref{eq:E1 terms'})
yields
\begin{equation}
|E_{1}|\leq\frac{1}{2}\sum_{i=1}^{d}Du_{x_{i}}A\cdot Du_{x_{i}}+C\psi(m)^{2}|p|^{2\gamma+1}+C|p|^{-1/2}k\zeta^{2}.\label{eq:E1 bound}
\end{equation}
Similarly, for the second error term, we use Young's Inequality and
(\ref{eq:signed 2'}) in (\ref{eq:E2 terms}), obtaining
\begin{multline}
|E_{2}|\leq\frac{1}{4}kDuA\cdot Du+C[|B_{m}||\Lambda D_{x}H||p|+|B_{m}||D_{x}H_{m}||H_{m}^{-1}||D_{x}H||p|+k|B_{m}|^{2}|D_{x}H_{m}|^{2}|H_{m}|^{-2}\\
+|D_{x}H|(|\Lambda B_{m}||p|+|B_{mm}||H_{m}|^{-1}|D_{x}H||p|)+k|D_{x}H|^{2}|B_{mm}|^{2}|H_{m}|^{-2}+|H_{m}||\Lambda\text{div}_{x}B||p|\\
+|\text{div}_{x}B_{m}||D_{x}H||p|+|\text{div}_{x}B_{m}|^{2}k+|\textup{div}_{x}B|(|\Lambda H_{m}||p|\\
+|H_{mm}||H_{m}|^{-1}|D_{x}H||p|)+k|\textup{div}_{x}B|^{2}|H_{mm}|^{2}|H_{m}|^{-2}+k|\utilde|(|H_{m}||\text{div}_{x}B|+|B_{m}||D_{x}H|)].\label{eq:E2 young}
\end{multline}
In view of (\ref{eq:HX HXX HXP bd}), (\ref{eq: Bm bd, DpBm bd,  DppB bd}),
(\ref{eq:Bxp bd}), and (\ref{eq:Bx Bxx Bxm bd}), we obtain

\[
|\Lambda D_{x}H|\leq C(|D_{xx}^{2}H|+k|D_{px}^{2}H|)\leq C\psi(m)(|p|^{\gamma_{2}}+|p|^{\gamma_{2}+1/2}),
\]
\[
|\Lambda\text{div}_{x}B|\leq C(|D_{xx}^{2}B|+k|D_{px}^{2}B|)\leq C\psi(m)(|p|^{\gamma_{2}-1}+|p|^{\gamma_{2}-1/2}),
\]
\[
|\Lambda B_{m}|\leq C(|D_{x}B_{m}|+k|D_{p}B_{m}|)\leq C\psi(m)(|p|^{\gamma_{2}-1}+|p|^{\gamma_{1}-1/2}),
\]
\[
|\Lambda\text{div}_{x}B|\leq C(|D_{xx}^{2}B|+k|D_{px}^{2}B|)\leq C\psi(m)(|p|^{\gamma_{2}-1}+|p|^{\gamma_{2}-1/2}).
\]
Consequently, (\ref{eq:E2 young}), (\ref{eq:Hm growth}), (\ref{eq:Hm bd}),
(\ref{eq: Bm bd, DpBm bd,  DppB bd}), (\ref{eq:HX HXX HXP bd}),
(\ref{eq:mHmx bd}), (\ref{eq:Bmm bd}), and (\ref{eq:Bx Bxx Bxm bd})
imply
\begin{equation}
|E_{2}|\leq\frac{1}{4}kDuA\cdot Du+C\psi(m)^{2}|p|^{2\gamma+1}.\label{eq:E2 bound}
\end{equation}
Having estimated the error terms, (\ref{eq:lineariz(v)}), (\ref{eq:signed 2''}),
(\ref{eq:E1 bound}), (\ref{eq:signed 2'}), and (\ref{eq:E2 bound})
yield
\begin{multline*}
L_{u}(v)=-kD\utilde A\cdot D\utilde-\sum_{i=1}^{d}Du_{x_{i}}A\cdot Du_{x_{i}}+E\\
\leq-\frac{k}{8}DuA\cdot Du-\frac{1}{C}k\zeta^{2}-\frac{1}{2}\sum_{i=1}^{d}Du_{x_{i}}A\cdot Du_{x_{i}}+C\psi(m)^{2}|p|^{2\gamma+1}+C|p|^{-1/2}k\zeta^{2}+Ck.
\end{multline*}
Therefore, in view of (\ref{eq:|p|^2 geq ||Dxu||^2}), (\ref{eq:signed 1}),
(\ref{eq:DpB}), and (\ref{eq:Hm growth}), 
\begin{multline*}
L_{u}(v)\leq\frac{k}{8}H_{m}pD_{p}B\cdot p-\frac{1}{2C}k|\zeta|^{2}+C\psi(m)^{2}|p|^{2\gamma+1}+C|p|^{-1/2}k\zeta^{2}+C|p|^{3/2}\\
\leq-\frac{1}{8C_{0}^{2}}\psi(m)^{2}|p|^{2\gamma+3/2}-\frac{\int1}{2C}k|\zeta|^{2}+C\psi(m)^{2}|p|^{2\gamma+1}+C|p|^{-1/2}k\zeta^{2}+C|p|^{3/2}\\
\leq-\psi(m)^{2}(\frac{1}{8C_{0}^{2}}|p|^{2\gamma+3/2}-C|p|^{2\gamma+1})-k\zeta^{2}(\frac{1}{2C}-C|p|^{-1/2})+C|p|^{3/2}.
\end{multline*}
So, given that $(x_{0},t_{0})$ is a maximum point of $v$, we have
$L_{u}(v)\geq0$, and, thus,

\begin{equation}
\psi(m)^{2}(\frac{1}{8C_{0}^{2}}|p|^{2\gamma+3/2}-C|p|^{2\gamma+1})+k\zeta^{2}(\frac{1}{2C}-C|p|^{-1/2})\leq C|p|^{3/2}.\label{eq:penult}
\end{equation}
This implies that either $\frac{1}{2C}-C|p|^{-1/2}\leq0$ or $\psi(m)^{2}(\frac{1}{8C_{0}^{2}}|p|^{2\gamma+3/2}-C|p|^{2\gamma+1})\leq C|p|^{3/2}$.
If the former holds, there is nothing to prove, so we may assume the
latter. We may further assume that $|p|$ is large enough that $\frac{1}{8C_{0}^{2}}|p|^{2\gamma+3/2}-C|p|^{2\gamma+1}\geq\frac{1}{16C_{0}^{2}}|p|^{2\gamma+3/2}$.
We thus obtain 
\begin{equation}
\psi(m)|p|^{\gamma}\leq C.\label{eq:fin}
\end{equation}
In view of (\ref{eq:upper coer}) and the fact that, by (\ref{eq:ut bound}),
$H$ is bounded below, we conclude that $m\leq C$. Hence $\psi(m)$
is bounded below, which finally yields $|p|\leq C$, concluding the
proof when $\gamma_{1}=\gamma$.

Now we describe the necessary changes in the proof to deal with the
case in which $\gamma_{1}<\gamma$. Setting 
\[
\eta=(2\gamma_{1}-\gamma+2)-\gamma_{2},
\]
we see, in view of (\ref{eq:condition exponents}), that $\eta>0$.
In (\ref{eq: k defi}), we replace $k$ by $k'=||D_{x}u||_{\QT}^{\kappa}$,
where 
\[
\kappa=\max\left(\frac{1}{2}(\gamma_{2}-\gamma)+1,\gamma_{1}-\gamma+\frac{3}{2}\right).
\]
The proofs of Case 1 and Case 2 follow through with no change until
the last step, leading in both cases to the inequality 
\begin{equation}
\frac{1}{C}|p|^{\gamma+\kappa}\leq C(1+|p|^{\gamma_{1}+1}+|p|^{\kappa}+|p|^{\gamma_{2}+1}).\label{eq:-2}
\end{equation}
By definition, $\kappa\geq\frac{3}{2}+\gamma_{1}-\gamma$, so the
left hand side of (\ref{eq:-2}) has higher degree than the right
hand side, and thus
\[
|p|\leq C.
\]
The proof of Case 3 proceeds analogously as well. (\ref{eq:E1 terms'})
and (\ref{eq:E2 young}) are obtained with no change. To estimate
the errors, instead of (\ref{eq:Lambda1}), (\ref{Lambda2}), and
(\ref{Lambda 3}), we now have the bounds
\[
|\Lambda Y^{+}|^{2},|\Lambda Y^{-}|^{2}\leq C(|D_{xp}^{2}H|^{2}+|D_{x}B_{m}|^{2}+k^{2}|D_{pp}^{2}H|^{2}+k^{2}|D_{p}B_{m}|^{2})\leq C\psi(m)^{2}(|p|^{2\gamma_{2}-2}+|p|^{2\gamma+2\kappa-4}),
\]
\[
|\Lambda H_{m}|^{2}\leq C|D_{x}H_{m}|^{2}+Ck^{2}|D_{p}H_{m}|^{2}\leq C\psi(m)^{2}m^{-2}(|p|^{2\gamma_{2}}+|p|^{2\gamma_{1}+2\kappa-2}),
\]
\[
|\Lambda D_{p}B|^{2}\leq C|D_{xp}^{2}B|^{2}+Ck^{2}|D_{pp}^{2}B|^{2}\leq C\psi(m)^{2}m^{2}(|p|^{2\gamma_{2}-4}+|p|^{2\gamma+2\kappa-6}).
\]
This allows us to estimate $E_{1}$ as before, this time obtaining
\begin{align}
|E_{1}| & \leq\frac{1}{2}\sum_{i=1}^{d}Du_{x_{i}}ADu_{x_{i}}+C\psi(m)^{2}(|p|^{2\gamma+2\kappa-2}+|p|^{2\gamma_{2}+\gamma-\gamma_{1}}+k|p|^{-(2+\gamma_{1}-\gamma-\kappa)}\zeta{}^{2}).\label{eq:E1 general}
\end{align}
Since the dominant power of $|p|$ in (\ref{eq:signed 1}) now has
the exponent 
\[
\alpha=\gamma+\gamma_{1}+\kappa,
\]
we must verify that (\ref{eq:E1 general}) does not have a higher
degree. Indeed,

\begin{multline}
\alpha-(2\gamma+2\kappa-2)=2+\gamma_{1}-\gamma-\kappa\geq\min(2+\gamma_{1}-\gamma-(\frac{1}{2}(\gamma_{2}-\gamma)+1),\gamma_{1}-\gamma+2-(\gamma_{1}-\gamma+\frac{3}{2}))=\frac{1}{2}\min(\eta,1),\label{alpha 1}
\end{multline}
\begin{equation}
\alpha-(2\gamma_{2}+\gamma-\gamma_{1})=2(\gamma_{1}-\gamma_{2})+\kappa\geq\gamma_{1}-\gamma_{2}+\kappa\geq\gamma_{1}-\gamma_{2}+\frac{1}{2}(\gamma_{2}-\gamma)+1=\frac{1}{2}\eta,\label{eq: alpha 2}
\end{equation}
Hence, letting $\epsilon=\frac{1}{2}\min(1,\eta)$, it follows from
(\ref{eq:E1 general}), (\ref{alpha 1}) and (\ref{eq: alpha 2}),
\begin{equation}
|E_{1}|\leq\frac{1}{2}\sum_{i=1}^{d}Du_{x_{i}}ADu_{x_{i}}+C\psi(m)^{2}(|p|^{\alpha-\epsilon}+k|p|^{-\epsilon}\zeta{}^{2}).\label{eq:E1 general F}
\end{equation}
Moving on to $E_{2}$, we first obtain 

\[
|\Lambda D_{x}H|\leq C(|D_{xx}^{2}H|+k|D_{px}^{2}H|)\leq C\psi(m)(|p|^{\gamma_{2}}+|p|^{\gamma_{2}-1+\kappa}),
\]
\[
|\Lambda\text{div}_{x}B|\leq C(|D_{xx}^{2}B|+k|D_{px}^{2}B|)\leq C\psi(m)(|p|^{\gamma_{2}-1}+|p|^{\gamma_{2}-2+\kappa}),
\]
\[
|\Lambda B_{m}|\leq C(|D_{x}B_{m}|+k|D_{p}B_{m}|)\leq C\psi(m)(|p|^{\gamma_{2}-1}+|p|^{\gamma_{1}-2+\kappa}),
\]
\[
|\Lambda\text{div}_{x}B|\leq C(|D_{xx}^{2}B|+k|D_{px}^{2}B|)\leq C\psi(m)(|p|^{\gamma_{2}-1}+|p|^{\gamma_{2}-2+\kappa}),
\]
and so, in place of (\ref{eq:E2 bound}),
\begin{align}
|E_{2}| & \leq\frac{1}{4}kDuA\cdot Du+C\psi(m)^{2}(|p|^{\gamma+\gamma_{2}}+|p|^{\gamma_{2}+\gamma+\kappa-1}+|p|^{2\gamma_{2}+\gamma-\gamma_{1}}+|p|^{\kappa+2\gamma_{2}-2+2\gamma-2\gamma_{1}}).\label{eq:E2 general}
\end{align}
We again verify that the exponents do not exceed $\alpha$,
\begin{equation}
\alpha-(\gamma+\gamma_{2})=\gamma_{1}-\gamma_{2}+\kappa\geq\gamma_{1}-\gamma_{2}+\frac{1}{2}(\gamma_{2}-\gamma)+1=\frac{1}{2}\eta,\label{eq: alpha 3}
\end{equation}
\begin{equation}
\alpha-(\gamma_{2}+\gamma+\kappa-1)=\gamma_{1}-\gamma_{2}+1\geq1,\label{eq: alpha 4}
\end{equation}
\begin{equation}
\alpha-(\kappa+2\gamma_{2}-2+2\gamma-2\gamma_{1})=(\gamma_{1}-\gamma_{2})+((2\gamma_{1}-\gamma+2)-\gamma_{2})\geq((2\gamma_{1}-\gamma+2)-\gamma_{2})=\eta.\label{eq: alpha 5}
\end{equation}
and thus, (\ref{eq:E2 general}), (\ref{eq: alpha 2}), (\ref{eq: alpha 3}),
(\ref{eq: alpha 4}), and (\ref{eq: alpha 5}) yield
\begin{equation}
|E_{2}|\leq\frac{1}{4}kDuA\cdot Du+C\psi(m)^{2}|p|^{\alpha-\epsilon}.\label{eq:E2 general f}
\end{equation}
Consequently, in view of (\ref{eq:E1 general F}) and (\ref{eq:E2 general f}),
we obtain, instead of (\ref{eq:penult}),

\[
\psi(m)^{2}(\frac{1}{8C_{0}^{2}}|p|^{\alpha}-C|p|^{\alpha-\text{\ensuremath{\epsilon}}})+k\zeta^{2}(\frac{1}{2C}-C|p|^{-\epsilon})\leq C|p|^{\kappa},
\]
and, thus, in place of (\ref{eq:fin}), this time we conclude

\begin{equation}
\psi(m)|p|^{\frac{\gamma+\gamma_{1}}{2}}\leq C.\label{eq:final}
\end{equation}
Since $\gamma_{1}<\gamma$, (\ref{eq:psi positive}) holds, and thus
we have (\ref{eq:-1}). This, together with (\ref{eq:final}), implies
that $|p|\leq C$, as wanted.
\end{proof}
The following lemma provides global, positive two-sided bounds for
the density in terms of the gradient bound.
\begin{lem}
\label{lem: m global bounds}Let $(u,m)\in C^{3}(\overline{Q_{T}})\times C^{2}(\overline{Q_{T}})$
be a solution to \textup{(\ref{eq:emfg})}, and set, for $K\in\mathbb{R},$
\[
\delta_{K}=\inf_{(x,p,s)\in\mathbb{T}^{d}\times\mathbb{R}^{d}\times(-\infty,K]}H^{-1}(x,p,s).
\]
 There exist constants $C,C_{1}>0$, with 
\[
C=C(C_{1},h(C_{1})),\;\;C_{1}=C_{1}(C_{0},||Du||_{C^{0}(\QT)},\delta_{||Du||}^{-1},||\psi||_{C^{0}[\delta_{||Du||},\infty)}),
\]
 such that 

\[
||m||_{C^{0}(\QT)}+||m^{-1}||_{C^{0}(\QT)}\leq C.
\]
\end{lem}

\begin{proof}
Due to (\ref{eq:lower coer}), $\delta_{K}>0$ is well-defined for
each $K\in\mathbb{R}$, and we may apply $H^{-1}(x,D_{x}u,\cdot)$
to both sides of the inequality

\[
H(x,D_{x}u,m)=u_{t}\leq||Du||_{C^{0}(\QT)},
\]
which yields, for $(x,t)\in\overline{Q_{T}}$,
\[
H^{-1}(x,D_{x}u,||Du||_{C^{0}(\QT)})\leq m(x,t).
\]
Letting $\delta=\delta_{||Du||}$, we thus obtain $\delta\leq m(x,t)$
and, hence,

\[
||m^{-1}||_{C^{0}(\QT)}\leq\delta^{-1}.
\]
On the other hand, 

\[
H(x,D_{x}u,m)-C_{0}\psi(m)|D_{x}u|^{\gamma}\geq u_{t}-C_{0}\psi(m)||D_{x}u||_{C^{0}(\QT)}^{\gamma}\geq-||Du||_{C^{0}(\QT)}-C_{0}||\psi||_{[\delta,\infty)}||D_{x}u||_{C^{0}(\QT)}^{\gamma}\geq-C_{1},
\]
which, in view of the definition of $h$ (see \ref{eq:h defi}), implies
that

\[
m\leq h(C_{1}).
\]
\end{proof}
We now summarize all of the a priori bounds obtained in this section.
\begin{thm}
\label{thm:Full C^1 a priori bound}Let $(u,m)\in C^{3}(\overline{Q_{T}})\times C^{2}(\overline{Q_{T}})$
be a solution to \textup{(\ref{eq:emfg})}, and let $\delta$ be defined
as in Lemma \ref{lem: m global bounds}. Then there exist constants
$M,M_{1},L,L_{1},K,K_{1}>0$, with
\[
L=\left(L_{1},g_{1}(f_{0}^{-1}(L_{1}))^{+},g_{0}(f_{1}^{-1}(-L_{1}))^{-},g_{0}^{-1}g_{1}(f_{0}^{-1}(L_{1})),\frac{1}{g_{1}^{-1}g_{0}(f_{1}^{-1}(-L_{1}))}\right),\;L_{1}=L_{1}(C_{0},T),
\]
\[
K=K(C_{1},||(\psi\circ h)^{-1}||_{C^{0}(0,K_{1}]}),
\]
\[
K_{1}=K_{1}(L,T^{-1},||\overline{C}||_{C^{0}[\frac{1}{L},L]},||\psi||_{C^{0}[\frac{1}{L},L]},||\psi^{-1}||_{C^{0}[\frac{1}{L},L]},||D_{x}g||_{C^{1}(\mathbb{T}^{d}\times[\frac{1}{L},L])},(2\gamma_{1}-\gamma+2-\gamma_{1})^{-1}),
\]

\[
M=M(M_{1},h(M_{1})),\;\;M_{1}=M_{1}(K,\delta_{K}^{-1},||\psi||_{C^{0}[\delta_{K},\infty)})
\]
such that
\[
||u||_{C^{0}(\overline{Q_{T}})}\leq L,\;\;||Du||_{C^{0}(\overline{Q_{T}})}\leq K,\;\;\text{and}\;\;||m||_{C^{0}(\QT)}+||m^{-1}||_{C^{0}(\QT)}\leq M.
\]
\end{thm}

\begin{proof}
This result follows from combining Lemma \ref{lem:aprioriu}, Lemma
\ref{lem: gradient a priori bound}, and Lemma \ref{lem: m global bounds}.
\end{proof}

\subsection{Classical solutions}

Having obtained the gradient bound, the existence result follows through
the method of continuity.
\begin{proof}[Proof of Theorem \ref{thm:smoothsols}.]
We only sketch the proof, which follows the same steps as \cite[Theorem 1.1]{Munoz}.
We define, for $\theta\in[0,1]$ and $(x,p,m)\in\mathbb{T}^{d}\times\mathbb{R}^{d}\times(0,\infty),$
\begin{align*}
H^{\theta}(x,p,m)=\theta H(x,p,m)+(1-\theta)H(0,p,m),\;\;\;B^{\theta}(x,p,m)=\theta B(x,p,m)+(1-\theta)B(0,p,m),\\
g^{\theta}(x,m)=\theta g(x,m)+(1-\theta)m,\;\;\;m_{0}^{\theta}(x)=\theta m_{0}(x)+(1-\theta),
\end{align*}
and consider the family of (\ref{eq:emfg}) systems
\begin{equation}
\tag{\ensuremath{\text{EMFG}_{\theta}}}\begin{cases}
-u_{t}+H^{\theta}(x,D_{x}u,m)=0 & (x,t)\in Q_{T},\\
m_{t}-\textrm{div}(B^{\theta}(x,D_{x}u,m))=0 & (x,t)\in Q_{T},\\
m(0,x)=m_{0}^{\theta}(x),\;u(x,T)=g^{\theta}(x,m(x,T)) & x\in\mathbb{T}^{d},
\end{cases}\label{eq:MFG theta}
\end{equation}
together with the corresponding elliptic and boundary operators $Q^{\theta}$
and $N^{\theta}$ associated to them, according to (\ref{eq:quasilinear}).
We observe first that for $\theta=0$, the solution is simply $(u,m)\equiv((t-T)H(0,0,1)+1,1).$
Now, by definition of $g^{\theta}$,

\begin{align*}
g_{0}^{\theta}\circ g_{0}^{-1}\circ g_{1} & =\theta g_{0}\circ g_{0}^{-1}\circ g_{1}+(1-\theta)g_{0}^{-1}\circ g_{1}\geq\theta g_{1}+(1-\theta)g_{0}^{-1}\circ g_{0}=g_{1}^{\theta},
\end{align*}
so we obtain

\begin{equation}
(g_{0}^{\theta})^{-1}g_{1}^{\theta}\leq g_{0}^{-1}g_{1},\label{eq:gtheta0}
\end{equation}
and similarly

\begin{equation}
g_{1}^{-1}g_{0}\leq(g_{1}^{\theta})^{-1}g_{0}^{\theta}.\label{eq:gtheta1}
\end{equation}
Moreover, setting 
\begin{align*}
f_{0}^{\theta}(m) & =\min_{\mathbb{T}^{d}}(-H^{\theta}(\cdot,0,m))=\theta f_{0}(m)-(1-\theta)H(0,0,m),\text{ and }\\
f_{1}^{\theta}(m) & =\max_{\mathbb{T}^{d}}(-H^{\theta}(\cdot,0,m))=\theta f_{1}(m)-(1-\theta)H(0,0,m),
\end{align*}
it is readily seen that, by definition,

\begin{equation}
(f_{0}^{\theta})^{-1}\leq f_{0}^{-1},\;\;\;f_{1}^{-1}\leq(f_{1}^{\theta})^{-1}.\label{eq:f theta}
\end{equation}
In view of (\ref{eq:gtheta0}), (\ref{eq:gtheta1}), and (\ref{eq:f theta}),
Theorem \ref{thm:Full C^1 a priori bound} yields a constant $C$,
independent of $\theta$, such that
\[
||u^{\theta}||_{C^{1}(\QT)}\leq C.
\]
Moreover, the classical $C^{1,\alpha}$ estimates for oblique derivative
problems (see, for instance, \cite[Lemma 2.3]{Lieberman}) yield
\begin{equation}
||u^{\theta}||_{C^{1+s'}(\QT)}\leq C.\label{eq:inter}
\end{equation}
for some $s'$. Now we define the Banach spaces
\[
E=C^{3,s}(\overline{Q_{T}}),\;F=C^{1,s}(\overline{Q_{T}})\times C^{2,s}(\partial Q_{T}),
\]
and the continuously differentiable operator $S:E\times[0,1]\rightarrow F$
by
\[
S(u,\theta)=(Q^{\theta}u,N^{\theta}u),\;(u,\theta)\in E\times[0,1].
\]
Direct computation shows that, for fixed $(u,\theta)\in E\times[0,1]$,
the linearization $S_{u}$ of $S$ with respect to $u$ has the form
$(L_{(u,\theta)}^{1}w,L_{(u,\theta)}^{2}w)$, where $L_{(u,\theta)}^{1}$
is a linear, uniformly elliptic operator and $L_{(u,\theta)}^{2}$
is a linear oblique boundary operator. Moreover, the homogeneous problem
$(L_{(u,\theta)}^{1}w,L_{(u,\theta)}^{2}w)=(0,0)$ has only the trivial
solution. The standard Fredholm alternative for linear oblique problems
(see \cite{GilbargTrudinger}) thus implies that $S_{u}$ is invertible
in $C^{3,s}(\overline{Q_{T}})$. Thus, by the implicit function theorem,
the set
\[
D=\{\theta\in[0,1]:\text{the equation \ensuremath{S(u,\theta)=(0,0)} has a unique solution \ensuremath{u\in C^{3,s}(\overline{Q_{T}})}}\}
\]
is open in $[0,1]$. On the other hand, (\ref{eq:inter}) together
with the Schauder estimates for linear oblique problems imply that
$D$ is also closed. Since $0\in D$, we conclude that $D=[0,1]$,
and in particular $1\in D$, which completes the proof.
\end{proof}

\subsection*{Acknowledgements}

The author would like to thank P. E. Souganidis for valuable discussions,
comments, and suggestions. He also thanks the anonymous referees for
their invaluable help in improving and clarifying the manuscript.
The author was partially supported by P.E. Souganidis\textquoteright s
National Science Foundation grant DMS-1900599, the Office for Naval
Research grant N000141712095 and the Air Force Office for Scientific
Research grant FA9550-18-1-0494.

{{\footnotesize   \bigskip   \footnotesize   \textsc{Department of Mathematics, University of Chicago, Illinois, 60637, USA}\\ \nopagebreak \textit{E-mail address}: \texttt{sbstn@math.uchicago.edu}}}
\end{document}